\documentclass[letterpaper, 10 pt, conference]{ieeeconf}  

\IEEEoverridecommandlockouts                              
\usepackage{cite}
\usepackage{amsmath,amssymb,amsfonts}
\usepackage{algorithmic}
\usepackage{graphicx}
\usepackage{textcomp}
\usepackage{stfloats} 
\usepackage{xcolor}
   
\newtheorem{theorem}{Theorem}

\newtheorem{remark}{Remark}

\newtheorem{lemma}{Lemma}
\newtheorem{assumption}{Assumption}

\title{\LARGE \bf
Lie bracket approximations of RMSprop for Extremum Seeking Control
}

\author{Victoria Grushkovskaya$^{1,3}$ and Christian Ebenbauer$^{2}$
\thanks{$^{1}$Department of Mathematics,
        University of Klagenfurt, 9020 Klagenfurt am W\"orthersee, Austria
        {\tt\small viktoriia.grushkovska@aau.at}}%
\thanks{$^{2}$Chair of Intelligent Control Systems, RWTH Aachen University, 52074 Aachen,
Germany        
        {\tt\small christian.ebenbauer@ic.rwth-aachen.de}}%
\thanks{$^{3}$Institute of Applied Mathematics \& Mechanics, NAS of Ukraine}
}

\begin{document}

\maketitle
\thispagestyle{empty}
\pagestyle{empty}

\begin{abstract}
This paper presents novel extremum seeking algorithms inspired by RMSProp dynamics and based on Lie bracket approximation techniques. To inherit the behavior of the classical continuous-time RMSProp algorithm, we introduce an extremum seeking scheme that exploits second-order Lie brackets, allowing the excitation of squared gradient terms in the resulting Lie bracket system.
Moreover, we propose two modified RMSProp-type algorithms that can be implemented using only first-order Lie bracket approximations, thereby reducing the complexity of the extremum seeking control design. The three resulting extremum seeking algorithms are compared through numerical simulations.
\end{abstract}

\section{INTRODUCTION}


RMSProp \cite{tieleman_hinton2012lecture} is a gradient-based optimization method with an adaptive step-size rule. The adaptation of the step size  is guided by the rule that a small absolute value of the (moving-averaged) gradient component requires a large step size component and vise versa. This rule is realized by filtering (moving-average) the squared components of the gradient and by a nonlinear scaling of the gradient. RMSProp is closely related to Adam \cite{kingma2014adam}; in fact, Adam can be interpreted as a momentum-based (stochastic) variant of RMSProp and is one of the main learning algorithms for deep neural networks. The RMSProp step size rule was primarily developed to improve the performance of gradient descent methods, specifically to address vanishing and exploding gradient problems in the training of deep neural networks. Moreover, RMSProp exhibits several desirable stability and robustness properties, see e.g. \cite{dimitrieski2026global}.

These properties make RMSProp an interesting candidate for extremum seeking control (ESC) algorithms as well. In their simplest form, ESC algorithms are gradient-free optimization algorithms that often rely on gradient-based algorithms and appropriate gradient approximation
schemes.   A first combination of ESC with RMSProp was presented in \cite{McN-Ahm-24}.
Traditionally, ESC algorithms are formulated in a continuous-time setting (see, e.g.,~\cite{Kr00,DurrAuto,Guay15,Nes09,Tan10,Sch24,yilmaz2025exponential,pokhrel2026higher}), which is also the framework adopted in this work. Hence, continuous-time RMSProp algorithms are of interest in this work and can be found, for example, in 
\cite{ma2022qualitative,bensaid_Lyapunov2024convergence,heredia2024modeling}. 
One of the main challenges of using RMSProp in ESC is that the gradient of the objective function enters in a nonlinear  fashion into the RMSProp vector field (iteration rule). As a consequence, ESC synthesis frameworks such as Lie bracket-based gradient approximation techniques \cite{DurrAuto} do not apply directly. 

The goal of this work is to present several alternative approaches to \cite{McN-Ahm-24} to synthesize RMSProp-based ESC algorithms  using Lie bracket approximation techniques, which, to the best of our knowledge, have not previously been applied to RMSProp-type dynamics.

The main contributions are as follows. First, we present an RMSProp-based ESC algorithm in which both the continuous-time gradient-based RMSProp vector field 
and the squared components of the gradient are approximated using Lie bracket approximation techniques.  In contrast to \cite{McN-Ahm-24}, our construction employs second-order Lie brackets to generate the squared gradient terms required by the classical RMSProp dynamics.
Second, we propose two novel continuous-time RMSProp algorithms that are particularly suited for first-order Lie bracket approximation techniques and thus are easier to implement. The first algorithm is based on a state extension, where the additional state corresponds to the low-pass filtered gradient. The second algorithm is based on a modified step-size rule and is, in its own right, an interesting alternative to the standard RMSProp \cite{tieleman_hinton2012lecture}. In particular, whereas in the standard RMSProp the gradient components are first squared and then filtered, the idea is to first filter the gradient components and then square them. Both modifications have the desired property that the gradient enters linearly in the corresponding vector fields, and thus standard first order Lie bracket approximation techniques apply.
Finally, simulations of the proposed methods are provided.

The remainder of this paper is organized as follows. 
In Section~\ref{sec_RMSProp}, we present a continuous-time formulation of RMSProp, analyze its stability properties, and introduce two modified versions of the algorithm. Section~\ref{sec_ES} is devoted to the development of three extremum seeking systems that inherit the qualitative properties of the RMSProp algorithms. The results are illustrated through numerical simulations in Section~\ref{sec_exmpls} and summarized in Section~\ref{sec_con}.


\emph{Notations:} $B_\delta(x^*)$,$\overline{B_\delta(x^*)}$ -- $\delta$-neighborhood of $x^*{\in} \mathbb R^n$ and its closure;\\
$C^\ell(D;\mathbb R)$ -- the set of all functions $h: D \to \mathbb{R}$ that are $\ell$-times continuously differentiable on $D$;\\
 for  $h{\in} C^1(\mathbb R^n;\mathbb R)$,  $\nabla h(x):=(\nabla_1 h(x),\dots,\nabla_n h(x))^\top$ with $\nabla_ih(x):=\frac{\partial h(x)}{\partial x_i}$;\\
 for $x\in\mathbb R$, $h'(x)=\frac{dh(x)}{dx}$;\\
for  $f,g \in C^1(\mathbb R^n;\mathbb R^n)$, 
the Lie bracket is 
 $[f,g](x)=\frac{\partial g(x)}{\partial x}f(x)-\frac{\partial f(x)}{\partial x}g(x)$.

 \section{RMSProp-like algorithms}~\label{sec_RMSProp}

\vspace{-0.5em}

In this section, we consider three continuous-time versions of RMSProp dynamics that will serve as reference systems for extremum seeking algorithms. In Section~\ref{sec_ES}, we introduce the corresponding extremum seeking systems and discuss the advantages and limitations of each approach.


A continuous-time version of classical RMSProp dynamics has been introduced in~\cite{ma2022qualitative}:
\begin{equation}\label{RMSProp}
\begin{aligned}
\dot x_i&=-\alpha_i\frac{\nabla_{i}J(x)}{\sqrt{v_i+\epsilon}},\\
\dot v_i&=-\beta_i \bigl(v_i-\nabla_{i}J(x)^2\bigr),
\end{aligned}
\end{equation}
where $\epsilon,\alpha_i ,\beta_i >0$, $i=1,\dots,n$.

In addition to the classical approach, we propose two modifications that are more convenient for extremum seeking implementation. The first one is obtained by introducing additional variables $\eta_i$ that provide  estimates of $\nabla_{i}J(x)$:
\begin{equation}
\label{RMSProp_mod}
\begin{aligned}
&\dot x_i=-\alpha_i \frac{\nabla_{i}J(x)}{\sqrt{v_i+\epsilon}},\\
&\dot v_i=-\beta_i (v_i-\eta_i^2),\\
&\dot \eta_i=-\gamma_i(\eta_i-\nabla_{i}J(x)),
\end{aligned}
\end{equation}
with positive constants $\epsilon,\alpha_i ,\beta_i ,\gamma_i$, $i=1,\dots,n$. 

The second modification is given by
\begin{equation}
\label{RMSProp_mod2}
\begin{aligned}
&\dot x_i=-\alpha_i \frac{\nabla_{i}J(x)}{\sqrt{v_i^2+\epsilon}},\\
&\dot v_i=-\beta_i (v_i-\nabla_{i}J(x)).
\end{aligned}
\end{equation}
We will show that the trajectories of system~\eqref{RMSProp} admit a second-order Lie bracket approximation. In contrast, for systems~\eqref{RMSProp_mod} and~\eqref{RMSProp_mod2}, a first-order Lie bracket approximation is sufficient. The reason for introducing system~\eqref{RMSProp_mod2} is to avoid increasing the system dimension.
We will discuss this in more detail in Section~III, while the rest of this section focuses on the stability properties.

\begin{theorem}\label{thm_stab}
\textit{Let $J\in C^1(\mathcal D;\mathbb R)$, $\mathcal D\subseteq \mathbb R^n$ be a domain, and let  there exist an $x^*\in \mathcal D$ such that:
\begin{enumerate}
\item[i)] $J(x)>J(x^*)$ for all $x\in \mathcal D\setminus\{x^*\}$; 
\item[ii)] $\nabla J(x)=0$ if and only if $x=x^*$.
\end{enumerate}
Then the following statements hold:
\begin{enumerate}
    \item the equilibrium  $q^*=({x^*}^\top,0^\top)^\top$ is  locally asymptotically stable  for system \eqref{RMSProp} w.r.t. the set $\mathcal  S=\{(x,v)\in\mathbb R^n\times\mathbb R^n:\ v_i\ge 0\}$;
    \item the equilibrium  $q^*=({x^*}^\top,0^\top,0^\top)^\top$ is  locally asymptotically stable  for system \eqref{RMSProp_mod} w.r.t. the set $\mathcal  S=\{(x,v,\eta)\in\mathbb R^n\times\mathbb R^n\times\mathbb R^n:\ v_i\ge 0\}$;
     \item the equilibrium  $q^*=({x^*}^\top,0^\top)^\top$ is  locally asymptotically stable  for system \eqref{RMSProp_mod2}.
\end{enumerate}
}
\end{theorem}


Due to space limitation, we do not include a formal proof.
However, for example, it is easy to see that in equation \eqref{RMSProp} the $x$-component
converges to a point where the gradient vanishes. Consequently, the $v$-component converges to zero. Such a cascade-like argument also applies to equation
\eqref{RMSProp_mod} and \eqref{RMSProp_mod2}.

\begin{remark}
To obtain global asymptotic stability, additional assumptions ensuring boundedness of trajectories are required. Typical sufficient conditions are radial unboundedness of 
$J$ or the global monotonicity condition $
\nabla J(x)^\top(x-x^*)\ge 0.
$
\end{remark}

\section{RMSProp-like ESC Algorithms}\label{sec_ES}
In this section, we introduce the extremum seeking algorithms whose trajectory approximate the trajectories of systems~\eqref{RMSProp}, \eqref{RMSProp_mod} and~\eqref{RMSProp_mod2}. According to the classical extremum seeking setting, we assume that only values of the cost function $J$ are available for control design, and use Lie bracket approximation techniques to approximate the gradient terms. %
{The results of this section establish practical asymptotic stability of the proposed extremum seeking systems, meaning that their trajectories converge to a neighborhood of $q^*$ whose radius   can be made arbitrarily small by increasing the excitation frequency (see, e.g.,~\cite{DurrAuto,GZE18}, for  definitions).}

\subsection{Extremum seeking system corresponding to~\eqref{RMSProp}}

\paragraph{Main idea}
To explain the main idea, we assume for the moment that $x\in\mathbb R$, i.e. system~\eqref{RMSProp} has the form
\begin{equation}\label{RMSProp_1D}
\begin{aligned}
\dot x&=-\alpha\frac{J'(x)}{\sqrt{v+\epsilon}},\\
\dot v&=-\beta \bigl(v-J'(x)^2\bigr).
\end{aligned}
\end{equation}
Since the gradient $J'(x)$ appears in both equations, we need introduce two extremum seeking loops. 
The $x$-equation can be approximated using standard Lie bracket approach~\cite{DurrAuto,GZE18} by noticing that 
{ \begin{equation}
    \label{bracket1}
    \left[
\begin{pmatrix}
     \phi_1\left(\frac{ J(x)}{\sqrt{v+\epsilon}}\right)\\
  0 
\end{pmatrix}{,} 
\begin{pmatrix}
  \phi_2\left(\frac{ J(x)}{\sqrt{v+\epsilon}}\right)\\
  0
\end{pmatrix}
\right]
{=}
\begin{pmatrix}
  {-}\frac{ J'(x)}{\sqrt{v+\epsilon}}\\
  0
\end{pmatrix},
\end{equation}
provided that $
[\phi_1,\phi_2]\equiv-1
$.}
However,
the presence of the quadratic term  $J'^2(x)$ in the second equation complicates the application of Lie brackets approximation technique. To address this challenge, we observe that
$$
\left[
\begin{pmatrix}
  J'\\
  0
\end{pmatrix}, 
\begin{pmatrix}
  0\\
  J
\end{pmatrix}
\right](x)
=
\begin{pmatrix}
  0\\
  J'(x)^2
\end{pmatrix}.
$$
Thus, for any functions $\psi_1,\psi_2$ such that $[\psi_1,\psi_2]\equiv1$,
\begin{equation}
  \left[
\left[
\begin{pmatrix}
   \psi_1\circ  J \\
  0
\end{pmatrix},
\begin{pmatrix}
  \psi_2\circ J \\
 0
\end{pmatrix}
\right],
\begin{pmatrix}
 0\\
  J 
\end{pmatrix}
\right](x)=
\begin{pmatrix}
  0\\
  J'(x)^2
\end{pmatrix}
.\label{bracket2}  
\end{equation}
Thus, to approximate the term $J'(x)^2$, we need to excite a second-order Lie bracket of the above form. This can be done by introducing three dither signals, for example, as in~\cite{GZ18}.
With that in mind, we introduce  the following extremum seeking system to follow the dynamics of~\eqref{RMSProp_1D}: \par
{\small\begin{align}  \label{ES_2order_1D}
    \dot {\bar x}=& \sqrt{2\alpha\kappa \omega }\Big(\phi_1\Big(\frac{  J(\bar x)}{\sqrt{\bar v +\epsilon}}\Big)\cos{(\kappa\omega t)} +\phi_2\Big(\frac{  J(\bar x)}{\sqrt{\bar v +\epsilon}}\Big)\sin{(\kappa \omega t)}\Big)\nonumber \\
    &\qquad+\sigma\omega^{2/3}\Big(\psi_{1}(J(\bar x))\cos({\varkappa_1\omega t})\nonumber+\psi_{2}(J(\bar x))\sin({\varkappa_2\omega t})\Big),\\
  \dot {\bar v }  =& -\beta  \bar v +\omega^{2/3}J(\bar x)\cos({\varkappa_1\omega t})\sin({\varkappa_2\omega t}),
  \end{align}} 
where  $\kappa,\varkappa_1,\varkappa_2{\in}\mathbb N$, $\varkappa_2^2>\varkappa_1^2$, $\sigma{=}\sqrt{4\beta (\varkappa_2^2{-}\varkappa_1^2)}$,
$
[\phi_1,\phi_2]\equiv -1$, $[\psi_1,\psi_2]\equiv1.
$
A generating formula for pairs of functions satisfying the above properties can be found in~\cite{GZE18}. We will also consider several specific choices in Section~\ref{sec_exmpls}.
%
The dithers 
$
u_1(t)=\sqrt{2\alpha\kappa}\cos{(\kappa\omega t)}$, $u_2(t)=\sqrt{2\alpha\kappa}\sin{(\kappa\omega t)}
$
 excite the first order Lie bracket from~\eqref{bracket1}~\cite{DurrAuto,GZ23}, 
while the dithers 
$$\begin{aligned}
&w_1(t)=\sigma\cos({\varkappa_1\omega t}), \, w_2(t)=\sigma\sin({\varkappa_2\omega t}), \\
&w_3(t)= \cos({\varkappa_1\omega t})\sin({\varkappa_2\omega t}).
\end{aligned}
$$
excite the second-order Lie bracket from~\eqref{bracket2}~\cite{GZ18}.
To ensure that there are no other first- and second-order Lie brackets generated by the above dithers, we refer to the following non-resonance assumptions~\cite{GZ18}:

\textit{The numbers $\kappa,\varkappa_1,\varkappa_2\in\mathbb N$ are pairwise distinct, $\kappa\notin\{\varkappa_1+\varkappa_2,|\varkappa_1-\varkappa_2|\}$, and ${\varkappa_1}/{\varkappa_2}\notin\{1,2,3,{1}/{2},{1}/{3}\}$.}

Next, we extend this idea to multi-variable setting and describe the stability properties of the resulting ES system.

\paragraph{Multi-variable case}

Denote by $q=(x^\top,v^\top)^\top\in\mathbb R^{2n}$ the state of the reference system~\eqref{RMSProp}, and by $\bar q=(\bar x^\top,\bar v^\top)^\top\in\mathbb R^{2n}$ the state of the corresponding extremum seeking system:
{\small \begin{equation}
    \label{ES_2order_q}
    \begin{aligned}
 \dot{\bar q}{=}f_0(\bar q)&{+}\omega^{1/2}\sum_{i=1}^{n}\sum_{j=1}^2f_{ij}(\bar q)u_{ij}(t){+}\omega^{2/3}\sum_{i=1}^{n}\sum_{j=1}^3g_{ij}(\bar q)w_{ij}(t),       
    \end{aligned}
\end{equation}}
where, for $i\in\{1,\dots,n\}$, $j=1,2$,
$$\begin{aligned}
&\ f_{ij}(\bar q)=\phi_j\Big(\tfrac{J(\bar x)}{\sqrt{\bar v_i+\epsilon}}\Big)e_i,\,g_{ij}(\bar q)=\psi_j(J(\bar x))e_i,\\
& g_{i3}(\bar q)=J(\bar x)e_{i+n},\,\,f_0(\bar q)=-(0,\dots,0,\beta_1\bar v_1,\dots,\beta_n\bar v_n)^\top,
\end{aligned}
$$
$e_j$ denotes the unit vector in $\mathbb R^{2n}$ with nonzero $j$-th entry.
The dithers
$$
u_{i1}(t)=\sqrt{2\alpha_i\kappa_i}\cos{(\kappa_i\omega t)},\,u_{i2}(t)=\sqrt{2\alpha_i\kappa_i}\sin{(\kappa_i\omega t)}
$$
excite the Lie brackets $\alpha_i[f_{i1},f_{i2}]$, and the dithers
$$
\begin{aligned}
&w_{i1}(t)=\sigma_i\cos({\varkappa_{i1}\omega t}), \, w_{i2}(t)=\sigma_i\sin({\varkappa_{i2}\omega t}), \\
&w_{i3}(t)=\cos({\varkappa_{i1}\omega t})\sin({\varkappa_{i2}\omega t})
\end{aligned}
$$
with $\sigma_i=\sqrt{4\beta_i (\varkappa_{i2}^2{-}\varkappa_{i1}^2)}$ exciting the Lie brackets $\beta_i[[g_{i1},g_{i2}],g_{i3}]$.
To ensure that no other Lie brackets  are excited, we extend the non-resonance assumption:
\begin{assumption}\label{ass_nonres} 
\textit{Let $ \varkappa_{i3}{ :=} \varkappa_{i1}{+}\varkappa_{i2}$, $\varkappa_{i4}{ :=} \varkappa_{i2}{-}\varkappa_{i1}>0$. We assume that, for all $i,j,k,\ell,m\in\{1,\dots,n\}$,
\begin{itemize}
    \item[i)]     $\varkappa_{i1}, \varkappa_{j2}, \varkappa_{k3}, \varkappa_{\ell4}$,  $\kappa_{m}$ are natural pairwise distinct;
    \item[ii)] there are no third order resonances between    $\varkappa_{i1}$, $\varkappa_{j2}$, $\varkappa_{k3}$,  $\varkappa_{\ell4}$, except those imposed by the definitions of $\varkappa_{k3}$ and $\varkappa_{\ell4}$. \\
\end{itemize}
}\end{assumption}
\begin{remark}
In the above control design, the parameters $\alpha_i,\beta_i$ are included in the dither coefficients, and the term $\frac{1}{\sqrt{\bar v_i+\epsilon}}$ is placed in the arguments of $\phi_j$. Alternative choices include incorporating $\alpha_i,\beta_i$ into $\phi_j$ and $\psi_j$, for example,
$\phi_{j}\Big(\frac{\alpha_i J(\bar x)}{\sqrt{\bar v_i+\epsilon}}\Big)$, 
$\psi_{j}(\beta_i J(\bar x))$,
or placing $\frac{1}{\sqrt{\bar v_i+\epsilon}}$ as a multiplier, e.g.,
$\frac{1}{\sqrt[4]{\bar v_i+\epsilon}}\phi_{j}\big(J(\bar x)\big).$
All these constructions lead to the same Lie bracket systems, although the resulting extremum seeking systems may exhibit different qualitative behavior, such as convergence rate or sensitivity to parameter choices.
\end{remark}

Before stating the main result of this subsection, we specify the assumptions imposed on the cost function $J$ and the  control vector fields. 

\begin{assumption}[Properties of the cost function]
\textit{Let $\mathcal D\subseteq\mathbb R^n$ be a domain, $x^*\in \mathcal D$, and $J\in C^2(\mathcal D;\mathbb R)$.
We assume
\begin{enumerate}
\item[i)] $J(x)>J(x^*)=:J^*$ for all $x\in\mathcal D\setminus\{x^*\}$;
\item[ii)] $\nabla J(x)=0$ if and only if $x=x^*$.
\end{enumerate}}
\end{assumption}
\begin{assumption}[Properties of control vector fields]
    \textit{The functions $\phi_{1},\phi_{2},\psi_{1},\psi_{2}\in C^2(\mathbb R;\mathbb R)$  satisfy}
\begin{center}
    $
[\phi_1,\phi_2](y)\equiv -1,\qquad [\psi_1,\psi_2](z)\equiv 1.
$
\end{center}
\end{assumption}
The following result describes the stability properties of~\eqref{ES_2order_q}.
\begin{theorem}\label{thm_RMSProp}
\textit{Consider system~\eqref{ES_2order_q} with cost function $J$, assume that Assumptions~1--3 are satisfied. Then  the point $q^*=({x^*}^\top,0^\top)^\top$ is  practically uniformly asymptotically stable provided that $\bar v_i(0)\ge 0$.}
\end{theorem}

The proof is in the Appendix. In particular, it shows that system~\eqref{ES_2order_q} approximates the dynamics of~\eqref{RMSProp} for sufficiently large values of~$\omega$. To achieve this, we introduced a novel construction based on second order Lie brackets to approximate the squared gradient term. This provides a nontrivial extension of Lie bracket approximation techniques and demonstrates their applicability to the realization of adaptive optimization dynamics. At the same time, several limitations of this approach should be noted. In particular, positivity of the variable $v_i$ is ensured only for sufficiently large values of $\omega$, which  gives more  restrictions for choosing it. Moreover, selecting frequencies that satisfy the non-resonance assumptions may become difficult, especially for multi-variable systems. Finally, the use of second order Lie brackets typically results in larger oscillations compared to first-order extremum seeking schemes. 
Unlike this approach, the RMSProp-like algorithms~\eqref{RMSProp_mod} and~\eqref{RMSProp_mod2} can be approximated using first order Lie brackets, as it will be show in the next subsections. 

\subsection{Extremum seeking system corresponding to~\eqref{RMSProp_mod}}

 To follow the dynamics of~\eqref{RMSProp_mod}, the corresponding gradient terms can be approximate with the first order Lie brackets. The main difficulty lies in approximating the term  $\nabla_iJ(x)$ in the equations for  $\dot{\bar\eta}$. This can be achieved, similarly to~\cite{Labar19}, by introducing dither signals in both the $\dot{\bar x}$ and $\dot{\bar\eta}$. We extend this idea by introducing a general class of control vector fields whose first-order Lie bracket yields the direction $
     \nabla_i J(x) e_{i+2n}$.
Namely, we denote $\bar q=(\bar x^\top,\bar v^\top,\bar\eta^\top)\in\mathbb R^{3n}$ and design the extremum seeking system as \par
{\small  \begin{equation}
    \dot{\bar q}=f_0(\bar q)+\sqrt\omega\sum_{\underset{j=1,2}{i=1}}^{n}\Big(f_{ij}(\bar q)u_{ij}(t)+\sum_{k=1,2}g_{ijk}(\bar q)w_{ijk}(t)\Big),\label{ES_1order}
\end{equation}}
$$
\begin{aligned}
\text{where }    f_0(\bar q)=-(0,\dots,0,\beta_1&(\bar v_1-\bar\eta_1^2),\\
    \dots,\beta_1&(\bar v_n-\bar\eta_n^2),\gamma_1\bar\eta_1,\dots,\gamma_n\bar\eta_n)^\top,
\end{aligned}
$$
the control vector fields  $f_{i1}$,  $f_{i2}$ and dithers $u_{1i}$, and $u_{2i}$ have the same form as in~\eqref{ES_2order_q}, with appropriate adjustments for the dimension:
 for $i\in\{1,\dots,n\}$, $j=1,2$, 
$
f_{ij}(\bar q)=\phi_j\Big(\frac{J(\bar x)}{\sqrt{\bar v_i+\epsilon}}\Big)e_i,
$
with $e_i$ denoting the unit vector in $\mathbb R^{3n}$ with nonzero $i$-th entry, and
$$
u_{i1}(t)=\sqrt{2\alpha_i\kappa_i}\cos{(\kappa_i\omega t)},\,u_{i2}(t)=\sqrt{2\alpha_i\kappa_i}\sin{(\kappa_i\omega t)}.
$$
Furthermore,  for $k=1,2$,
$$
\begin{aligned}
&g_{i1k}(\bar q)=\psi_{1k}\big(J(\bar x)\big)e_{i},\,g_{i2k}(\bar q)=\psi_{2k}\big(J(\bar x)\big)e_{i+2n},\\
&w_{i1k}=\sqrt{2\gamma_i\varkappa_{ik}}\cos{(\varkappa_{ik}\omega t)},\,w_{i2k}=\sqrt{2\gamma_i\varkappa_{ik}}\sin{(\varkappa_{ik}\omega t)},
\end{aligned}
$$
where $\kappa_i,\varkappa_{jk}\in\mathbb N$ are pairwise distinct, and the following assumption holds:

\begin{assumption}[Properties of the control vector fields]
\textit{The functions
$\phi_{1},\phi_2,\psi_{11},\psi_{12},\psi_{21},\psi_{22}\in C^2(\mathbb R;\mathbb R)$
satisfy
$
[\phi_1,\phi_2]\equiv -1,
$
and
$\displaystyle
\psi_{11}\psi_{21}'+\psi_{12}\psi_{22}'\equiv 1.
$}
\end{assumption}

Under the above assumptions, the dithers $w_{i11}$, $w_{i12}$, $w_{i21}$, $w_{i22}$ excite the combination of first order Lie brackets
$\gamma_i([g_{i11},g_{i21}]+[g_{i12},g_{i22}])$, or, equivalently, 
{\small $$
\begin{aligned}
\gamma_i\big[(\psi_{11}\circ J)e_{i},&(\psi_{21}\circ J)\big)e_{i+2n}\big](\bar x)\\
&+\gamma_i\big[(\psi_{12}\circ J)e_{i},(\psi_{22}\circ J)\big)e_{i+2n}\big](\bar x)\\
=\gamma_i(\psi_{11}\psi_{21}'&+\psi_{12}\psi_{22}')\nabla_i{J(\bar x)}e_{i+2n}=\gamma_i\nabla_i{J(\bar x})e_{i+2n}.
\end{aligned}
$$
}
A possible option for the choice of $ \psi_{11},\psi_{12},\psi_{21},\psi_{22}$ is
$\psi_{11}=1$, $\psi_{21}=J$, $\psi_{12}=\psi_{22}=0$,
as it was done in~\cite{Labar19}, or
$ \psi_{11}=-\psi_{22}=\psi_1$, $\psi_{21}=\psi_{12}=\psi_2,$
with
$
[\psi_1,\psi_2]\equiv 1.
$

Similarly to the previous subsection, the stability properties of~\eqref{ES_1order} are described in the following theorem.
\begin{theorem}\label{thm_RMSProp_mod}
\textit{Consider system~\eqref{ES_1order} with cost function $J$, assume that Assumptions~2 and~4 are satisfied. Then  the point $q^*=({x^*}^\top,0^\top,0^\top)^\top$ is  practically uniformly asymptotically stable provided that $\bar v_i(0)\ge 0$.}
\end{theorem}

The proof of the theorem is similar to the proof of Theorem~2. The main difference is that it uses the Chen--Fliess series expansion up to first-order terms.
Moreover, since the $\bar v_i$-component admits the representation
$$
\bar v_i(t)=\bar v_i^0e^{-\beta_i t}+\beta_i \int_0^t \bar\eta(s)^2e^{-\beta_i (t-s)}ds,
$$
it remains positive for  $t\ge0$ provided $\bar v_i^0\ge0$, thus, there is no need to choose $\omega$ to ensure this requirement. However, the additional equation increases the dimension of the system.

\subsection{Extremum seeking system corresponding to~\eqref{RMSProp_mod2}}

Similarly to the previous subsection, the third algorithm~\eqref{RMSProp_mod2} can be realized as an extremum seeking system without exciting second order Lie brackets. Unlike the previous construction, its dimension is the same as that of the original algorithm~\eqref{RMSProp}. The resulting system is similar to~\eqref{ES_1order}, but without the $\dot{\bar\eta}$
 subsystem, i.e., with  $\bar q=(\bar x^\top,\bar v^\top)\in\mathbb R^{2n}$:
\par
{\small \begin{equation}
    \dot{\bar q}=f_0(\bar q)+\sqrt\omega\sum_{i=1}^{n}\Big(\sum_{j=1}^2f_{ij}(\bar q)u_{ij}(t)+\sum_{j,k=1}^2g_{ijk}(\bar q)w_{ijk}(t)\Big),\label{ES_1order_mod}
\end{equation}}
where
$
    f_0(\bar q)=-(0,\dots,0,\beta_1\eta_1,\dots,\beta_n\eta_n)^\top,
$
the control vector fields  $f_{i1}$,  $f_{i2}$ and dithers $u_{1i}$, and $u_{2i}$ are the same as in~\eqref{ES_2order_q}, and $g_{ijk}(\bar q)$, $w_{ijk}$ have the same form as in~\eqref{ES_1order}, with appropriate adjustments for the dimension:
 for $i\in\{1,\dots,n\}$, $j,k=1,2$,
$$
\begin{aligned}
&g_{i1k}(\bar q)=\psi_{1k}\big(J(\bar x)\big)e_{i},\,g_{i2k}(\bar q)=\psi_{2k}\big(J(\bar x)\big)e_{i+n},\\
&w_{i1k}=\sqrt{2\gamma_i\varkappa_{ik}}\cos{(\varkappa_{ik}\omega t)},\,w_{i2k}=\sqrt{2\gamma_i\kappa_{ik}}\sin{(\kappa_{ik}\omega t)},
\end{aligned}
$$
where $\kappa_i,\varkappa_{jk}\in\mathbb N$ are pairwise distinct, and $e_j$ denotes the unit vector in $\mathbb R^{2n}$ with nonzero $j$-th entry.
Then the following statement holds. 
\begin{theorem}\label{thm_RMSProp_mod2}
\textit{Consider system~\eqref{ES_1order_mod} with cost function $J$, assume that Assumptions~2 and~4 are satisfied. Then  the point $q^*=({x^*}^\top,0^\top)^\top$ is  practically uniformly asymptotically stable.}
\end{theorem}
The proof follows the same lines as that of Theorem~2 and relies on the Chen--Fliess series expansion truncated to first-order terms, involving only first-order Lie brackets. In contrast to the approaches based on~\eqref{RMSProp} and~\eqref{RMSProp_mod}, the denominator $\sqrt{v_i^2+\epsilon}$ is non-zero by  construction, which eliminates additional well-definiteness issues.
\begin{figure*}[!hb]
\includegraphics[width=1\linewidth]{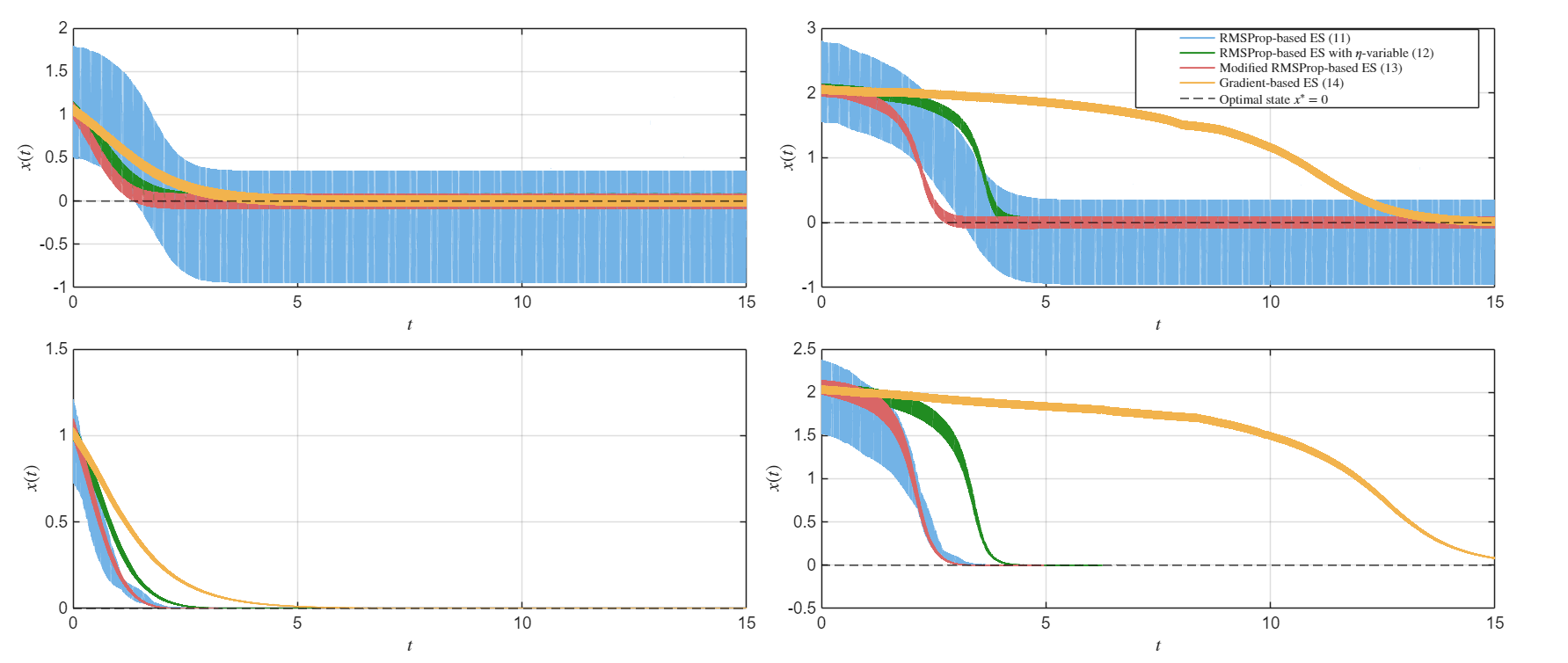}
  \caption{Time plots of the $x$-component of the solutions of systems~\eqref{ex1a}--\eqref{ex1d} with the cost function~\eqref{ex1cost}, parameters~\eqref{parameters1} and control vector fields defined by~\eqref{bounded} (top plots) and~\eqref{vanish} (bottom plots). The solutions are initialized at $x(0)=1$ (left plots) and $x(0)=2$ (right plots), with $v(0)=0.5$, $w(0)=0.5$.}
    \label{fig_ex1}
    \end{figure*}
\section{Numerical simulations}\label{sec_exmpls}
In this section, we compare all three extremum seeking approaches, as well as the classical gradient-based extremum seeking scheme from~\cite{DurrAuto,GZE18}, using numerical simulation. 
In all the consider cases, we put $\alpha_i{=}\beta_i{=}\gamma_i{=}1$, $i\in\{1,\dots,n\}$.
For convenience of reading, we present all four extremum seeking systems here. Namely, the extremum seeking system corresponding to the classical RMSProp algorithm~\eqref{RMSProp_1D} has the form~\eqref{ES_2order_q},\par
{\footnotesize
\begin{equation} 
\label{ex1a}
\begin{aligned}
  \dot {\bar x}_i=& \sqrt{2\alpha_i\kappa_i \omega}\Big(\phi_1\Big(\frac{  J(\bar x)}{\sqrt{\bar v_i +\epsilon}}\Big)\cos{(\kappa_i\omega t)} +\phi_2\Big(\frac{  J(\bar x)}{\sqrt{\bar v_i +\epsilon}}\Big)\sin{(\kappa_i \omega t)}\Big) \\
    &\qquad+\sigma_i\omega^{2/3}\Big(\psi_{1}(J(\bar x))\cos({\varkappa_{1i}\omega t})+\psi_{2}(J(\bar x))\sin({\varkappa_{2i}\omega t})\Big),\\
  \dot {\bar v }_i  =& -\beta_i  \bar v_i +\omega^{2/3}J(\bar x)\cos({\varkappa_{1i}\omega t})\sin({\varkappa_{2i}\omega t}).
\end{aligned}
    \end{equation}}
The  extremum seeking systems~\eqref{ES_1order} and~\eqref{ES_1order_mod}:
\par
\par
{\footnotesize
\begin{equation}\label{ex1b}
\begin{aligned}
&\dot{\bar x}_i
=
\sqrt{2\alpha_i\omega\kappa_i}
\big(\phi_1\Big(\frac{J(\bar x)}{\sqrt{\bar v_i +\epsilon}}\Big)\cos(\kappa_i\omega t)+\phi_2\Big(\frac{J(\bar x)}{\sqrt{\bar v_i +\epsilon}}\Big)\sin(\kappa_i\omega t)\big) \\
&
+\sqrt{2\gamma_i\omega}\Big(
\sqrt{\varkappa_{i1}}\,\psi_{11}(J(\bar x))\cos(\varkappa_{i1}\omega t)
+\sqrt{\varkappa_{i2}}\,\psi_{12}(J(\bar x))\cos(\varkappa_{i2}\omega t)
\Big),\\ 
&\dot{\bar v}_i
=
-\beta_i\big(\bar v_i-\bar \eta_i^2\big),\\ 
&\dot{\bar \eta}_i
=
-\gamma_i \bar \eta_i
+\sqrt{2\gamma_i\omega}\Big(
\sqrt{\varkappa_{i1}}\,\psi_{21}(J(\bar x))\sin(\varkappa_{i1}\omega t)\\
&\quad\quad\quad\qquad+\sqrt{\varkappa_{i2}}\,\psi_{22}(J(\bar x))\sin(\varkappa_{i2}\omega t)
\Big),
\end{aligned}
\end{equation}
}
and \par
{\footnotesize
\begin{equation}\label{ex1c}
\begin{aligned}
\dot{\bar x}_i
&=
\sqrt{2\alpha_i\omega\kappa_i}
\big(\phi_1\Big(\frac{J(\bar x)}{\sqrt{\bar v_i^2 +\epsilon}}\Big)\cos(\kappa_i\omega t)+\phi_2\Big(\frac{J(\bar x)}{\sqrt{\bar v_i^2 +\epsilon}}\Big)\sin(\kappa_i\omega t)\big) \\
&\quad
+\sqrt{2\beta_i\omega}\Big(
\sqrt{\varkappa_{i1}}\,\psi_{11}(J(\bar x))\cos(\varkappa_{i1}\omega t)
\\
&\quad\quad\quad\qquad +\sqrt{\varkappa_{i2}}\,\psi_{12}(J(\bar x))\cos(\varkappa_{i2}\omega t)
\Big),\\ 
\dot{\bar v}_i
&=
-\beta_i \bar v_i 
+\sqrt{2\beta_i\omega}\Big(
\sqrt{\varkappa_{i1}}\,\psi_{21}(J(\bar x))\sin(\varkappa_{i1}\omega t)\\
&\quad\quad\quad\qquad+\sqrt{\varkappa_{i2}}\,\psi_{22}(J(\bar x))\sin(\varkappa_{i2}\omega t)
\Big).
\end{aligned}
\end{equation}
}
Finally, the ES system from~\cite{DurrAuto,GZE18} has the form\par
{\footnotesize\begin{equation}
\label{ex1d}
\dot{\bar x}_i= \sqrt{2\alpha_i \omega\kappa_i }\Big(\phi_1({J(\bar x)})\cos(\kappa_i\omega t) +\phi_2({J(\bar x)})\sin(\kappa_i \omega t)\Big).
\end{equation}
}

\subsection{Single-variable cost function}

Consider the function $J:\mathbb{R}\to[0,1)$ given by
\begin{equation}
    \label{ex1cost}
    J(x)=\frac{1}{2}\big(1-e^{-x^2}\big).
\end{equation}
For $|x|\leq 1$, the function $J$ behaves like $x^2$, and all algorithms exhibit similar qualitative behavior (cf.~Fig.~\ref{fig_ex1}, left). However, outside this region the function becomes flat, since $J(x)\to 1$ as $|x|\to\infty$. For initial conditions in this region, one observes that RMSProp-like algorithms outperform the classical gradient-based extremum seeking approach (cf.~Fig.~\ref{fig_ex1}, right).

Note that, similarly to~\cite{GZE18}, one can choose control vector fields satisfying different properties, such as bounded update rates (based on the idea originally introduced in~\cite{Sch14}):
\begin{equation}
    \label{bounded}
    \begin{aligned}
     &\phi_1(z)=\psi_{2}(z)=\psi_{12}(z)=\psi_{21}(z)=\sin z,\\
     &\phi_2(z)=\psi_1(z)=\psi_{11}(z)= -\psi_{22}(z)=\cos z.
\end{aligned}
\end{equation}
The results of numerical simulations with these functions are presented in the top plots of Fig.~\ref{fig_ex1}. Furthermore, for positive definite cost functions, the classical asymptotic stability property can be achieved with the following functions (see also~\cite{SD17}):
\begin{equation}
    \label{vanish}
    \begin{aligned}
     &\phi_1(z)=\psi_{2}(z)=\psi_{12}(z)=\psi_{21}(z)=\sqrt z\sin\ln z,\\
     &\phi_2(z)=\psi_1(z)=\psi_{11}(z)= -\psi_{22}(z)=\sqrt z\cos\ln z.
\end{aligned}
\end{equation}
This property is visualized in the bottom plots of Fig.~\ref{fig_ex1}.
In all these cases, we put  
\begin{equation}
    \label{parameters1}
    \begin{aligned}
    &\kappa_1=1,\,\varkappa_{11}=3,\,\varkappa_{12}=8,\,\epsilon=10^{-6}, \,\omega=2\pi\times 10^2.\\
    \end{aligned}
\end{equation}

\subsection{Quadratic form with elliptic level sets}
To illustrate the applicability of the developed approaches to multi-variables cases, we consider the cost function
\begin{equation}
    \label{ex2cost}
    J(x)=x_1^2+\frac{x_2^2}{20}.
\end{equation}
As seen in the previous example, the solutions of the extremum seeking system~\eqref{ex1a} exhibit highly oscillatory behavior, even in the single-variable case. In this example, we focus on systems~\eqref{ex1b} and~\eqref{ex1c}. Note that 

For numerical simulations, we put $x(0)=1$ and
\begin{equation}
    \label{parameters2}
    \begin{aligned}
    \kappa_1{=}1,\,\varkappa_{11}{=}5&,\,\varkappa_{12}{=}7,\,\kappa_2{=}3,\,\varkappa_{21}{=}9,\,\varkappa_{22}{=}11,\\
     &\epsilon=10^{-6}, \,\omega=2\pi\times 10^2.\\
    \end{aligned}
\end{equation}
The time plots of the $x$-components of the solutions to systems~\eqref{ex1b} and~\eqref{ex1c}, together with that of system~\eqref{ex1d}, are shown in Fig.~\ref{fig_ex2}. The plots correspond to the control vector fields defined by~\eqref{bounded} (left) and~\eqref{vanish} (right).
\begin{figure*}[!ht]
    \begin{minipage}{0.5\linewidth}
    \end{minipage}\hfill
      \begin{minipage}{1\linewidth}
      \begin{center}
        \includegraphics[width=1\linewidth]{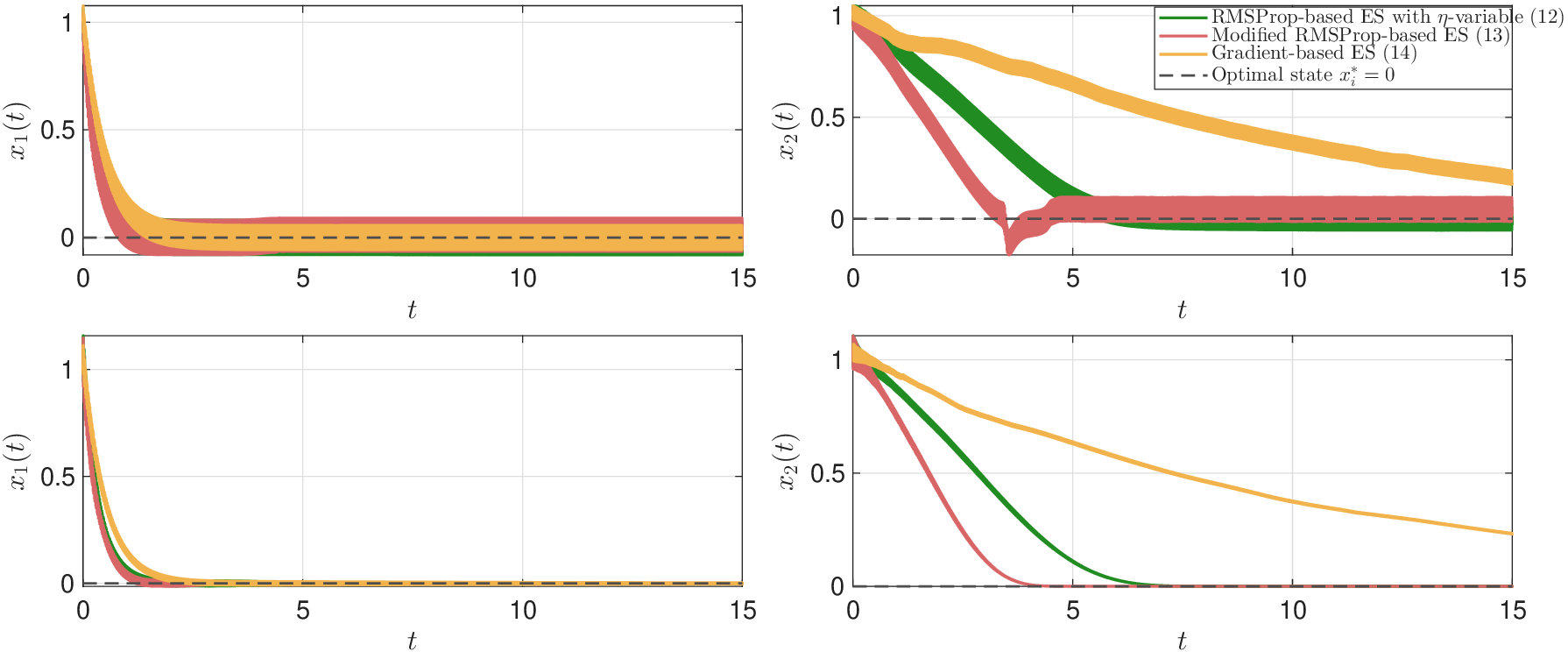}
        \end{center}
    \end{minipage}
  \caption{Time plots of the $x_1$ (left) and $x_2$ (right) components of the solutions of systems~\eqref{ex1a}--\eqref{ex1d} with the cost function~\eqref{ex2cost}, parameters~\eqref{parameters2}, and control vector fields defined by~\eqref{bounded} (left) and~\eqref{vanish} (right). The solutions are initialized at $x(0)=(1,1)^T$  with $v(0)=(0.5,0.5)^\top$, $w(0)=(0.5,0.5)^\top$.}
    \label{fig_ex2}
    \end{figure*}

\section{Conclusions}
\label{sec_con}
In this paper, we have introduced three novel extremum seeking algorithms based on the integration of RMSProp-like dynamics and Lie bracket approximation techniques. The presence of the squared gradient components represents a key challenge for Lie bracket approximation of RMSProp, since these terms cannot be directly approximated using standard tools based on the first order Lie brackets. The three proposed approaches address this difficulty in different ways and provide different trade-offs between approximation of the original RMSProp dynamics and the complexity of the extremum seeking implementation.

The scheme~\eqref{ES_2order_q} allows us  to approximate the trajectories of the classical RMSProp algorithm, however, it requires excitation of second-order Lie brackets in order to generate the squared gradient terms in the reference system. While this result and the corresponding stability analysis provide a nontrivial extension of higher-order Lie bracket approximation techniques~\cite{pokhrel2026higher,GE_CDC25}, the choice of excitation frequencies becomes challenging in the multivariable case due to the non-resonance requirements. This may lead to larger excitation frequencies and increased amplitudes of oscillations in the extremum seeking dynamics. 

In view of this, we have proposed two modifications of RMSProp-type algorithms that do not require second-order Lie brackets for their extremum seeking implementation.
The first modification~\eqref{ES_1order} avoids second-order Lie brackets at the cost of introducing an additional state variable to filter the squared gradient terms, and therefore increases the system dimension.
The second modification~\eqref{ES_1order_mod}  preserves the original system dimension but changes the RMSProp update rule. Unlike classical RMSProp, where the gradient is first squared and then filtered, system~\eqref{RMSProp_mod2} first filters the gradient and then squares this estimate in the denominator. In particular, for slowly varying $x$, the auxiliary variable $v_i$ in classical RMSProp tracks $(\nabla_i J(x))^2$, whereas in~\eqref{RMSProp_mod2} it tracks $\nabla_i J(x)$, so that $v_i^2\approx(\nabla_i J(x))^2$. Thus, both schemes provide approximately the same adaptive scaling of the $x_i$-dynamics in the quasi-steady state. Therefore, the modified scheme~\eqref{RMSProp_mod2} retains the essential RMSProp-type adaptive scaling mechanism while allowing the use of first-order Lie bracket approximations.
It is worth noting that the modified RMSProp algorithm~\eqref{RMSProp_mod2} is of independent interest, as it demonstrates improved convergence properties in simulations and does not require non-negativity of the $v_i$-components. We leave a detailed comparison with the classical RMSProp algorithm~\eqref{RMSProp} for future work.

\bibliographystyle{ieeetr}
\bibliography{biblio_ES}

\section*{Appendix. Proof of Theorem~\ref{thm_RMSProp}}
\begin{proof}
The proof follows the same conceptual framework as classical results in extremum seeking theory based on Lie bracket approximation techniques.

\emph{Step 0. Preliminary constructions. }
Assume  $\delta_x,\delta_v > 0$ are such that $\overline{B_{\delta_x}(x^*)}\subset{\rm int}\mathcal D$, and $\delta_v\in(0,\epsilon)$. 
We choose a $\delta'_{x}>\delta_x$ such that $\mathcal D':=  \overline{B_{\delta'_x}(x^*)}\subseteq \mathcal D$ and 
define $\mathcal L_c=\{x\in \mathcal D': J(x)-J^*\le c\}$. 
Let
$c_0=\sup\limits_{x\in\overline{B_{\delta_x}(x^*)}}(J(x)-J^*)>0$,
and  assume $\delta'_x$ is chosen sufficiently large to ensure  $\mathcal L_{c_0}\subset   {\rm int} \mathcal D'$. Then we have the following set inclusion:
\begin{equation}
    \label{sets}
    \overline{B_{\delta_x}(x^*)}\subseteq \mathcal L_{c_0}\subset   {\rm int} \mathcal D'\subseteq \mathcal D.
\end{equation}
 Obviously,  $\mathcal L_{c} \subseteq \mathcal L_{c_0}$ for each $c\le c_0$.

We also fix an $m_v\in (\delta_v,\epsilon)$ and introduce the following  notations:\\
$\mathcal S_0= \mathcal L_{c_0} \times \overline{B_{\delta_v}(0)}$,
$\mathcal S'=\mathcal D'\times [-m_v, M_v]^n$,\\
where $[-m_v, M_v]^n=\underbrace{[-m_v, M_v]\times\dots\times [-m_v, M_v]}_{n\text{ times}}$, $M_v>0$,
$M_\phi=\sup\limits_{\bar q\in \mathcal S'}\max\limits_{i=1,\dots,n,j=1,2}\Big|\phi_j\Big(\frac{J(\bar x)}{\sqrt{\bar v_i+\epsilon}}\Big)\Big|$,\\
$M_\psi=\sup\limits_{\bar x\in \mathcal D'}\max\limits_{j=1,2}|\psi_j( J(\bar x))|$,
$M_J=\sup_{\bar x\in \mathcal D'}|J(\bar x)|$,\\
$M_{1J}=\sup_{\bar x\in \mathcal D'}\|\nabla J(\bar x)\|,$
$M_{2J}=\sup_{\bar x\in \mathcal D'}\|\nabla^2J (\bar x)\|$.

\emph{Step 1. Well-definiteness in $\mathcal S'=\mathcal D'\times [-m_v, M_v]^n$}.

 As a first step, we  show that  there exists an $\omega_0 > 0$ such that, for all $\omega > \omega_0$, the solutions of system~\eqref{ES_2order_q} with the initial data $\bar q(0)=\bar q^0 \in \mathcal S_0$  remain in the compact set $\mathcal S'$ for  $t \in [0,T_\omega]$ with $T_\omega=2\pi\omega^{-1}$ and $M_v\in(\delta_v,\infty)$ to be specified later.

Using the integral representation of solutions of system~\eqref{ES_2order_q} with initial condition $\bar q^0\in \mathcal S_0$ for the $\bar x$-component, and the variation of constants formula for the $\bar v_i$-component, we conclude that for all $t\in[0,T_\omega]$,
$$
\begin{aligned}
&\|\bar x(t)-\bar x^0\|\le  4\pi\omega^{-1/2}M_\phi\Big({\sum_{i=1}^n\alpha_i\kappa_i}\Big)^{1/2}\\
&\qquad \qquad\qquad{ +} 2\sqrt 2\pi\omega^{-1/3}M_\psi\Big({\sum_{i=1}^n\sigma_i^2}\Big)^{1/2} ,\\
&|\bar v_i(t)- \bar v_i^0e^{-\beta_i t}|\le \omega^{-1/3}2\pi M_J\sigma_i.
\end{aligned}
$$
For simplicity,   assume  $\omega>1$. 
Then, for all  $t\in[0,T_\omega]$, 
\begin{equation}
    \label{apriori}
\begin{aligned}
&\|\bar x(t)-\bar x^0\|\le c_x \omega^{-1/3} ,\|\bar v_i(t)-\bar v_i^0e^{-\beta_i t}|\le c_v \omega^{-1/3},
\end{aligned} 
\end{equation}
where $c_x {=}4\pi M_\phi\Big({\sum_{i=1}^n\alpha_i\kappa_i}\Big)^{1/2}
{+} 2\sqrt 2\pi M_\psi\Big({\sum_{i=1}^n\sigma_i^2}\Big)^{1/2}$, $c_v =2\pi M_J\max_{i=1,\dots,n}\sigma_i$.

Denote  
$d_x={\rm dist}(\mathcal L_{c_0},\partial \mathcal D')>0$, and let
 $\omega_0{=}\max\Big\{1,\Big(\frac{c_x }{d_x}\Big)^3,\Big(\frac{c_v}{m_v-\delta_v}\Big)^3,\Big(\frac{c_v }{M_v-\delta_v}\Big)^3\Big\}$. Then, for  $\omega{>}\omega_0$,
\begin{equation}
    \label{bounds}
\begin{aligned}
&\|\bar x(t)-\bar x^0\|\le  c_x \omega^{-1/3}\le d_x,\\
&\bar v_i(t) \ge -\delta_v- c_v \omega^{-1/3}\ge -m_v,\\
&\bar v_i(t)\le  \delta_v +  c_v \omega^{-1/3} \le M_v,
\end{aligned} 
\end{equation}
which implies  $\bar q(t)\in\mathcal S'$ for all $t\in[0,T_\omega$ whenever $\bar q^0\in\mathcal S_0$.

\emph{Step 2. Estimation for $J(\bar x(T_\omega))$}.

The second step is to estimate the value  $J(x(T_\omega))$. For this purpose, we employ the Chen--Fliess series expansion. Similarly to~\cite[Proof of Theorem 1]{GZ18} but  taking into account the drift term $f_0$, we  represent the solutions of ~\eqref{ES_2order_q} as follows:
\begin{equation} \label{chen_q}
\begin{aligned}
     \bar q(T_\omega)=\bar q^0&+2\pi\omega^{-1}\big( f_0(\bar q^0)+\alpha_i\sum_{i=1}^n[f_{i1},f_{i2}](\bar q^0)\\
     &+\sum_{i=1}^n\beta_i [[g_{i1},g_{i2}],g_{i3}](\bar q^0)\big) +R_\omega, 
\end{aligned}
\end{equation}
where $R_\omega=(R_{\omega,x}^\top,R_{\omega,v}^\top)^\top$, $\|R_\omega\|\le C_R\omega^{-4/3}$ for all $\omega>\omega_0$, $\bar q^0\in \mathcal S'$, with some positive constant  $C_R$.  Straightforward computations show that, for all $i\in\{1,\dots,n\}$,
$$
\begin{aligned} &[f_{i1},f_{i2}](\bar q)=-\frac{ \nabla_iJ(\bar x)}{\sqrt{\bar v_i+\epsilon}}e_i,\\ 
&[[g_{i1},g_{i2}],g_{i3}](\bar q)=\nabla_iJ(\bar x)^2e_{i+n}, 
\end{aligned}
$$
therefore, the representation~\eqref{chen_q} can be written as
\begin{equation}
    \label{chen_2nd}
    \begin{aligned}
  \bar q(T_\omega)=\bar q^0-2&\pi\omega^{-1}\sum_{i=1}^n\Big(\frac{\alpha_i\nabla_iJ(\bar x^0)}{\sqrt{\bar v_i^0+\epsilon}}e_i\\
&+\beta_i (\bar v_i^0- \nabla_i J(\bar x^0)^2)e_{i+n}+R_\omega.      
    \end{aligned}
\end{equation}
Using  Taylor's formula, we write
$$
\begin{aligned}
   J(\bar x(T_\omega))=J&(\bar x^0)+\nabla J(\bar x^0)(\bar x(T_\omega)-\bar x^0)\\
   &+\frac{1}{2}(\bar x(T_\omega)-\bar x^0)^\top\nabla^2J(\theta)(\bar x(T_\omega)-\bar x^0),
\end{aligned}
$$
with some $\theta$ on the segment joining $ \bar x^0$ and $\bar x(T_\omega)$. With representation~\eqref{chen_q}, this yields
{\small $$
\begin{aligned}
J(\bar x(T_\omega))=& J(\bar x^0)-\nabla J(\bar x^0)\Big(2\pi \omega^{-1}\sum_{i=1}^n \frac{\alpha_i\nabla_iJ(\bar x^0)}{\sqrt{\bar v_i^0+\epsilon}}e_i-R_{\omega,x}\Big)\\
&+\frac{1}{2}\nabla^2 J(\theta)\Big(2\pi \omega^{-1}\sum_{i=1}^n \frac{\alpha_i\nabla_iJ(\bar x^0)}{\sqrt{\bar v_i^0+\epsilon}}e_i-R_{\omega,x}\Big)^2\\
\end{aligned}
$$}
{\small $$
\begin{aligned}
\le J(\bar x^0)-2\pi &\omega^{-1}\sum_{i=1}^n \frac{\alpha_i\nabla_iJ(\bar x^0)^2}{\sqrt{\bar v_i^0+\epsilon}}\Big(1-\frac{M_{2J}\alpha_i \omega^{-1}}{\sqrt{\bar v_i^0+\epsilon}} \Big)+R_J,
\end{aligned}
$$}
where $R_J{=}\|\nabla J(\bar x^0)R_{\omega,x}\|{+}\|\nabla^2J(\theta)\|R_{\omega,x}^2{\le} C_{R_J}\omega^{-{4/3}}$ with $C_{R_J}= M_{1J}C_R +M_{2J}C_R^2\omega_0^{-4/3}$, for all $\omega>\omega_0$.
Thus,
{\small $$
\begin{aligned}
J(\bar x(T_\omega))\le J(\bar x^0)&-2\pi \omega^{-1}\sum_{i=1}^n \frac{\alpha_i\nabla_iJ(\bar x^0)^2}{\sqrt{\bar v_i^0+\epsilon}}\Big(1-\frac{M_{2J}\alpha_i \omega^{-1}}{\sqrt{\bar v_i^0+\epsilon}} \Big)\\
&+C_{R_J}\omega^{-4/3}.    
\end{aligned}
$$}
In particular, for any $\bar v_i^0\in[ -\delta_v,M_v]$, $\tilde\lambda\in(0,1)$, and for any $\omega>\omega_1=\frac{M_{2J}}{(1-\tilde\lambda)\sqrt{\epsilon-\delta_v}}\max\limits_{i=1,\dots,n}\{\alpha_i\} $,
\begin{equation}
    \label{chen_J}
J(\bar x(T_\omega))\le J(\bar x^0)-\lambda\omega^{-1}\|\nabla J(\bar x^0)\|^2+C_{R_J}\omega^{-4/3}
\end{equation}
where $\lambda=\frac{2\pi  \tilde\lambda}{\sqrt{\epsilon+M_v}}\min\limits_{i=1,\dots,n}\{ \alpha_i\}$.

\emph{Step 3. Practical convergence of $\bar x$.}
In the third step, we aim to prove that for any $\rho>0$ there exists an $\omega_2>0$ and $t_f\ge 0$ such that, for all $\omega>\omega_2$, $t\ge t_f$, $\bar x(t)\in\mathcal L_{\rho}$.

 Assume  that  $\bar v_i(t)\in [-m_v,M_v]$ for all $t\ge 0$ (we will prove this later). 
For any given $\rho\in(0,c_0)$, define $\rho_1>0$ such that $J(x)-J^*\le \rho/2$ for all $x\in\{x\in \mathcal D':\|\nabla J(x)\|\le \rho_1\}$.
Consider two cases.

\textit{Case 1.} $\|\nabla J(\bar x^0)\|\ge \rho_1$. Then representation~\eqref{chen_J} implies that, for any $\hat \lambda\in(0,\lambda)$, $\omega>\omega_{2,1}=\Big(\frac{C_{R_J}}{(\lambda-\hat\lambda)\rho_1^2}\Big)^3$,
$$
\begin{aligned}
  J(\bar x(T_\omega))-J^*&\le J(\bar x^0)-J^*\\
  &\qquad\quad -\omega^{-1}\|\nabla J(\bar x^0)\|^2\Big(\lambda-\frac{C_{R_J}\omega^{-1/3}}{\rho_1^2}\Big)\\
  &\le J(\bar x^0)-J^*-\hat\lambda\omega^{-1}\rho_1^2\le c_0,  
\end{aligned}
$$
which means that   $\bar x(T_\omega)\in \mathcal L_{c_0}$ provided that $\bar x^0\in \mathcal L_{c_0}$ . 

\textit{Case 2.} $\|\nabla J(\bar x^0)\|< \rho_1$. Then representation~\eqref{chen_J}  and definition of $\rho_1$ imply  that, for any $\omega>\omega_{2,2}=\Big(\frac{2C_{R_J}}{\rho}\Big)^{3/4}$,
$$
\begin{aligned}
J(\bar x(T_\omega))-J^*\le &J(\bar x^0)-J^*+C_{R_J}\omega^{-4/3} \le \frac{\rho}{2}+\frac{\rho}{2}=\rho,
\end{aligned}
$$
and we may consider again cases 1) and 2) depending on the value of $\|\nabla J(\bar x(T_\omega))\|$.
Assuming $\bar v_i(t)\in [-m_v,M_v]$ for all $t\ge 0$, we may repeat all the previous steps and conclude that, after  some finite time $t_f\ge0$, $\bar x(t)$ enters the set $\mathcal L_\rho $ and stays there for all $t\ge t_f$.

\emph{Step 4. Well-definiteness of $\bar v_i(t)$ on $[0,+\infty)$}

On the fourth step, we prove that $\bar v_i(t)\in [-m_v,M_v]$ for all $t\ge 0$ and specify $M_v$. 

First, let us show that $\bar v_i(t)\ge -m_v$ for all $t\ge 0$ provided that $\bar v_i^0\in [-\delta_v,\delta_v]$.  
Let $\omega_{3,1}=2\pi \max\limits_{i=1,\dots,n}\beta_i $. Then for all $\omega>\omega_{3,1}$, 
representation~\eqref{chen_q} yields 
$$
\begin{aligned}
\bar v_i(T_\omega)&=\bar v_i^0(1-2\pi \beta_i \omega^{-1})+2\pi\beta_i \omega^{-1} \nabla_i J(\bar x^0)^2+R_{\omega,v_i}\\
&\ge \bar v_i^0(1-2\pi \beta_i \omega^{-1})- C_R\omega^{-4/3}\\
&\ge -\delta_v(1-2\pi \beta_i \omega^{-1})- C_R\omega^{-4/3}.
\end{aligned}
$$
Hence, taking $\omega>\omega_{3,2}=\max\limits_{i=1,\dots,n}\Big(\frac{C_R}{2\pi \delta_v\beta_i }\Big)^3$, we ensure
$
\bar v_i(T_\omega) \ge -\delta_v.
$
Together with estimate~\eqref{bounds}, this yields
$$
\bar v_i(t)\ge  -m_v\text{ for all }t\in[0,2T_\omega].
$$
For ensuring that $\bar v_i(t)\le M_v$, we use again representation~\eqref{chen_q}, which implies
$$
\bar v_i(T_\omega) \le (1-2\pi \beta_i \omega^{-1})\bar v_i^0+2\pi \beta_i \omega^{-1}J'(\bar x^0)^2+C_R\omega^{-4/3}.
$$
From the results of Step 3, this gives 
$$
\bar v_i(T_\omega) -\nu^2 \le (1-2\pi \beta_i \omega^{-1})(\bar v_i^0-\nu^2)+C_R\omega^{-4/3},
$$
where $\nu=\max\{\sup\limits_{x\in \mathcal L_{c_0}}\|\nabla J(x)\|,\rho\}$.
Iterating the above estimate, we obtain that, for any $k\in\mathbb N$,
$$
\bar v_i(kT_\omega) -\nu^2 \le \zeta^k(\bar v_i^0-\nu^2)+C_R\omega^{-4/3}\sum_{j=0}^{k-1}\zeta_i^j,
$$
with $\zeta_i=1-2\pi \beta_i \omega^{-1}\in(0,1)$.
The last term in the above estimate gives the geometric sum, which can be estimated as 
$$
\sum_{j=0}^{k-1}\zeta_i^j=\frac{1-\zeta_i^{k}}{1-\zeta_i}=\frac{1-(1-2\pi \beta_i \omega^{-1})^{k}}{2\pi \beta_i \omega^{-1}}\le \frac{\omega}{2\pi \beta_i }.
$$
Thus,
$$
\bar v_i(kT_\omega) -\nu^2 \le \zeta^k(\bar v_i^0-\nu^2)+\frac{C_R}{2\pi \beta_i }\omega^{-1/3},
$$
and
$$
\bar v_i(kT_\omega) \le \max\{\bar v_i^0,\nu^2\}+\frac{C_R}{2\pi \beta_i }\omega^{-1/3}.
$$
Taking into account~\eqref{apriori}, we conclude
$$
\bar v_i(t)\le \max\{\delta_v,\nu^2\}+\Big(\frac{C_R}{2\pi \beta_i }+c_v \Big)\omega^{-1/3}
$$
Given any $M_v\in\big(\max\{\delta_v,\nu^2\},\infty\big)$, let $\omega_{3,3}=\max\limits_{i=1,\dots,n}\Big(\frac{C_R/(2\pi \beta_i )+c_v }{M_v-\max\{\delta_v,\nu^2\}}\Big)^3$. Then for any $\omega>\omega_{3,3}$, $\bar v_i^0\in (-\delta_v,\delta_v)$, $i=1,\dots,n$, we have $\bar v_i(t)\le M_v$.

Thus, for any $\omega>\omega_3=\max\{\omega_{3,1},\omega_{3,2},\omega_{3,3}\}$,
$$
\bar v_i(t)\in [-m_v,M_v]\text{ for all }t\ge 0.
$$

\emph{Step 5. Final conclusions} 

Summarizing Steps~1--4 and noting that the parameter $\rho>0$ in Step~3 can be chosen arbitrarily small, we conclude that for all $\omega>\max\{\omega_0,\omega_1,\omega_2,\omega_3\}$, the point
$q^*=(x^*,0)^\top
$
is practically uniformly asymptotically stable for system~\eqref{ES_2order_q}.
Indeed, Step~3 shows that for any $\rho>0$ there exists a time $t_f\ge0$ such that
$\|\bar x(t)-x^*\|\le \rho
\quad \text{for all } t\ge t_f.$
Moreover, Step~4 implies that
$|\bar v_i(t)|
\le
 \hat \nu^2 + C\omega^{-1/3},
\quad t\ge t_f,$
where
$\hat \nu=\max\{\delta_v,\sup_{x\in\mathcal L_\rho}\|\nabla J(x)\|\}
$
and $C>0$ is a constant independent of $\omega$.

Since $\nabla J(x^*)=0$ and $\nabla J$ is continuous, we have $\nu\to0$ as $\rho\to0$. Hence,
for any $\varepsilon>0$ one can choose $\rho>0$ sufficiently small and $\omega$ sufficiently large so that
$$\nu^2 + C\omega^{-1/3} < \varepsilon/2
\quad\text{and}\quad
\rho < \varepsilon/2.$$
Then for all $t\ge t_f$,
$$\|\bar q(t)-q^*\|
\le
\|\bar x(t)-x^*\|+\|\bar v(t)\|
<
\varepsilon.$$
Furthermore, Steps~1 and~4 guarantee that solutions starting sufficiently close to $q^*$ remain in a bounded neighborhood of $q^*$ for all $t\ge0$, which establishes practical uniform stability. Consequently, the point $q^*$ is practically uniformly asymptotically stable for system~\eqref{ES_2order_q}.
 \end{proof}
\end{document}